\documentclass[11pt, oneside,reqno]{amsart}
\usepackage{amsfonts, amstext, amsmath, amsthm, amscd, amssymb}
\usepackage{manfnt,mathrsfs,mathtools}
\usepackage{url}
\usepackage[nobysame, alphabetic]{amsrefs}
\usepackage{subcaption}

\usepackage{tikz}
\usetikzlibrary{cd, arrows}
\usepackage[arrow,matrix,curve,cmtip,ps]{xy}
\usepackage{enumerate}

\usepackage{graphicx, pinlabel, xcolor}
\definecolor{darkblue}{rgb}{0,0,0.6}
\usepackage[breaklinks, pdftex, ocgcolorlinks,colorlinks=true, citecolor=darkblue, filecolor=darkblue, linkcolor=darkblue, urlcolor=darkblue]{hyperref}
\usepackage[capitalize,noabbrev]{cleveref}
\usepackage[paper=letterpaper, text={160mm,220mm},centering]{geometry}

\makeatletter
\newtheorem*{rep@theorem}{\rep@title}
\newcommand{\newreptheorem}[2]{%
	\newenvironment{rep#1}[1]{%
		\def\rep@title{#2 \ref{##1}}%
		\begin{rep@theorem}}%
		{\end{rep@theorem}}}
\makeatother

\newtheorem{proposition}{Proposition}[section]
\newtheorem{theorem}[proposition]{Theorem}

\newreptheorem{theorem}{Theorem}
\newreptheorem{lemma}{Lemma}
\newreptheorem{proposition}{Proposition}
\newreptheorem{corollary}{Corollary}

\theoremstyle{definition}
\newtheorem{definition}[proposition]{Definition}
\newtheorem{question}[proposition]{Question}

\theoremstyle{remark}
\newtheorem{remark}[proposition]{Remark}

\newtheorem*{remark*}{Remark}

\numberwithin{equation}{section}

\newcommand{\Z}{\mathbb{Z}}

\newcommand{\sm}{\smallsetminus}

\DeclareMathOperator{\Wh}{\mathrm{Wh}}

\begin{document}
\title{Not all knots are smoothly round handle slice}
	
\author{Tye Lidman}
\address{North Carolina State University, Raleigh NC, USA}
\email{tlid@math.ncsu.edu}

\author{Allison N.~Miller}
\address{Swarthmore College, Swarthmore PA, USA}
\email{amille11@swarthmore.edu}

\author{Arunima Ray}
\address{The University of Melbourne, Melbourne VIC, Australia}
\email{aru.ray@unimelb.edu.au}

\def\subjclassname{\textup{2020} Mathematics Subject Classification}
\expandafter\let\csname subjclassname@1991\endcsname=\subjclassname

\begin{abstract}
    Freedman and Krushkal showed that if the surgery conjecture and the $s$-cobordism conjecture hold for all topological 4-manifolds, then every link with pairwise zero linking numbers is topologically \emph{round handle slice}. Kim, Powell, and Teichner showed that every knot is topologically round handle slice. We show that infinitely many knots fail to be smoothly round handle slice.
\end{abstract}

\maketitle

\section{Introduction}\label{sec:intro}

Two of the main open problems in $4$-dimensional topology are the topological surgery conjecture and the topological $s$-cobordism conjecture. These  conjectures are known to be true for $4$-manifolds with so-called good fundamental groups~\cite{FQ}*{Theorems~7.1A and~11.3A}, but are open in general. Given a link $L\subseteq S^3$, Freedman and Krushkal~\cite{freedman-krushkal:engel} defined a compact, smooth $4$-manifold~$R(L)$ and an associated link $\gamma(L)\subseteq \partial R(L)$. They showed that if both the topological surgery and $s$-cobordism conjectures hold in general, then for every $L$ with vanishing pairwise linking numbers, the associated link $\gamma(L)$ bounds a collection of locally flat, pairwise disjoint, embedded discs in $R(L)$. When this is true, we say that $L$ is \emph{topologically round handle (RH) slice} (see~\cref{def:RHslice}). Finding links which fail to be topologically RH-slice is an attractive possible route to disproving the surgery and $s$-cobordism conjectures. 

A simpler proof for Freedman--Krushkal's theorem was presented by Kim, Powell, and Teichner in \cite{kim-powell-teichner:RHP}. A consequence of their argument is that every knot is topologically RH-slice. The main result of this short note is that, in contrast, not all knots are smoothly RH-slice.

\begin{theorem}\label{thm:main}
    Let $K$ be the $(-2,4m-1)$-cable of the torus knot $T_{2,6m+1}$ for $m \geq 1$. Then $K$ is not smoothly round handle slice. 
\end{theorem}

The proof consists of showing that if a knot is smoothly RH-slice, then one of a list of associated integer homology spheres must bound an integer homology ball (\cref{prop:ZHB-bound}). 
Then we use invariants from Heegaard--Floer homology to obstruct this for the knots mentioned in the statement above.

\subsection*{Acknowledgements} 
We thank the University of Glasgow and the Centre de Recherches Mathematiques for their hospitality, as well as Jen Hom and Adam Levine for helpful conversations.

\section{Definitions}\label{sec:definitions}

An $n$-dimensional \emph{round $k$-handle} is a copy of $S^1\times D^k\times D^{n-k-1}$ attached along $S^1\times \partial D^k\times D^{n-k-1}$. In particular, given a pair of disjoint, embedded circles $C_1,C_2$ in $\partial X$ for $X$ an $n$-manifold, a round $1$-handle attached along $C_1$ and $C_2$  identifies tubular neighbourhoods of $C_1$ and $C_2$ with respect to given framings. Round handles were introduced by Asimov in~\cite{asimov:round-handles}.

Let $L=L_1\sqcup \dots \sqcup L_n\subseteq S^3$ be an $n$-component oriented ordered link. Let $X_L\coloneqq S^3\sm \mathring{\nu}L$ denote the \emph{exterior} of $L$, i.e.~the complement of an open tubular neighbourhood of $L$. Let $\mu_i\subseteq X_L$ be an oriented meridian of $L_i$, the $i$th component of $L$, and let $\lambda_i\subseteq X_L$ denote a $0$-framed oriented longitude of $L_i$, pushed further into $X_L$ as indicated on the left of Figure~\ref{fig:R(K)} when $L$ is a knot.
Let $\nu(\mu_i),\nu(\lambda_i)\subseteq X_L$ denote closed  tubular neighbourhoods of $\mu_i$ and $\lambda_i$.
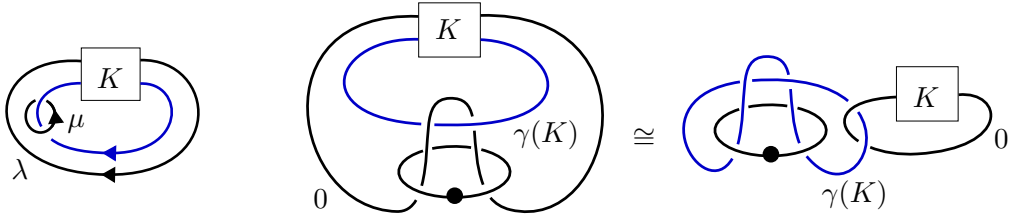
\begin{figure}[htb]
    \centering
    \begin{tikzpicture}
\node[anchor=south west, inner sep=0] at (0,.4)
{\includegraphics[height=1.8cm]{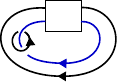}};
\node[anchor=south west,inner sep=0] at (4,0)
{\includegraphics[height=2.8cm]{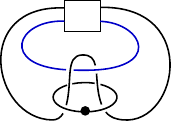}};
\node at (1.4,1.8){$K$};
\node at (.2,.6){$\lambda$};
\node at (.95, 1.2){$\mu$};
\node at (4.2,.2){$0$};
\node at (7.2,1.05){$\gamma(K)$};
\node at (11.3, .25){$\gamma(K)$};
\node at (8.5,1){$\cong$};
\node at (5.9, 2.45){$K$};
\node at (12.25, 1.55){$K$};
\node at (13.25,1){$0$};
\node[anchor=south west,inner sep=0] at (9,.5)
{\includegraphics[height=1.6cm]{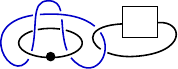}};
\end{tikzpicture}
    \caption{A knot $K$ with meridian $\mu$ and longitude $\lambda$ (left). Attaching a round handle identifying $\mu$ and $\lambda$ gives $R(K)$, with $\gamma(K) \subseteq \partial R(K)$ (center). We slide $\gamma(K)$ over the 2-handle to obtain an alternate diagram (right). }
    \label{fig:R(K)}
\end{figure}
\begin{definition}\label{def:RHslice}
    Let $L=L_1\sqcup \dots \sqcup L_n\subseteq S^3$ be an $n$-component oriented ordered link. Define the manifold $R(L)$ to be the result of attaching $4$-dimensional round $1$-handles to $B^4$, such that the $i$th round handle is attached along $\nu(\mu_i)\sqcup \nu(\lambda_i)\subseteq X_L\subseteq S^3$ for $i=1,\dots,n$, using the $0$-framing. The image of $L$ in $\partial R(L)$ after this procedure is called $\gamma(L)$. 
    The link $L$ is said to be \emph{smoothly round handle (RH) slice} if $\gamma(L)$ is smoothly slice in~$R(L)$, i.e.~ if it bounds a pairwise disjoint collection of smoothly embedded discs in $R(L)$. If the discs are only locally flat embedded, then $L$ is said to be \emph{topologically round handle (RH) slice}.
\end{definition}

\begin{remark}
Note that 
$R(L)\cong X_0(L)\,\natural \,n(S^1\times D^3),$
where $X_0(L)$ is the trace of the $0$-framed Dehn surgery on $S^3$ along all the components of $L$. See \cref{fig:R(K)} for a Kirby diagram when $L$ is a knot, adapted from~\cite{kim-powell-teichner:RHP}.
\end{remark}

\begin{remark}
    Every knot $K$ bounds a PL disc in $B^4$, and this disc survives the round handle attachment to become a PL disc in $R(K)$ with boundary $\gamma(K)$. So every knot is PL RH-slice.
\end{remark}

\section{Proofs}\label{sec:proofs}
\begin{proposition}\label{prop:ZHB-bound}
    Let $K\subseteq S^3$ be a knot which is smoothly round handle slice. Then either
    \begin{enumerate}
        \item $S^3_{1/n}(K)$ is the boundary of a smooth integer homology ball for some $n \neq 0$, or
        \item $S^3_1(\Wh^+(K))$ is the boundary of a smooth integer homology ball.
    \end{enumerate}
\end{proposition}

\begin{proof}
    Suppose that $\gamma(K)$ bounds a smoothly embedded disc $D$ in $R(K)$, and denote the complement of an open tubular neighbourhood of $D$ by $X_D\coloneqq R(K)\sm \mathring{\nu}(D)$. We see that $X_D$ is a smooth, compact $4$-manifold with  $\partial X_D\eqqcolon Y(K,m_K)$ for some $m_K\in \Z$, where $Y(K,m_K)$ is the $3$-manifold shown in \cref{fig:boundary}. 
    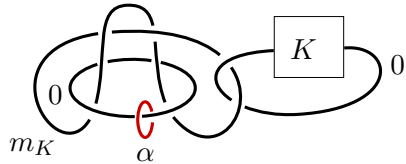
\begin{figure}[htb]
    \centering
     \begin{tikzpicture}
\node[anchor=south west,inner sep=0] at (0,0)
{\includegraphics[height=1.8cm]{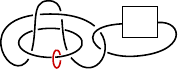}};
\node at (3.6,1.2){$K$};
\node at (1.5,-0.2){$\alpha$};
\node at (4.85,1){$0$};
\node at (0,-.1){$m_K$};
\node at (.3,.6){$0$};
\end{tikzpicture}
    \caption{A surgery diagram for $\partial X_D\eqqcolon Y(K,m_K)$, with a  curve $\alpha \subseteq \partial X_D$. The integer $m_K$ arises because we do not know how $D$ frames $\gamma(K)$. 
    }
    \label{fig:boundary}
\end{figure}
   
    Using the Mayer--Vietoris sequence, we see that 
    \[
    \widetilde{H_i}(X_D;\Z)\cong \begin{cases}
        \Z\langle\alpha\rangle  &i=1;\\
        0   &\textrm{otherwise,}
    \end{cases}
    \]
    where $\alpha$ is the curve shown in \cref{fig:boundary}. 
   We now have two cases.

    \textbf{Case 1: $m_K\neq 0$.} Let $W$ be obtained from $X_D$ by  attaching a $0$-framed $2$-handle to $\alpha\subseteq \partial X_D$. By construction, $W$ is an integer homology ball, and we show via Kirby calculus (\cref{fig:mkneq0}) that  $\partial W=S^3_{-1/m_K}(K)$.

\begin{figure}[htb]
    \centering
         \begin{tikzpicture}
\node[anchor=south west,inner sep=0] at (0,0)
{\includegraphics[height=1.8cm]{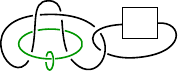}};
\node[anchor=south west,inner sep=0] at (6.2,.25)
{\includegraphics[height=1.5cm]{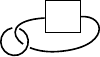}};
\node at (1.4,-0.2){$0$};
\node at (4.7,1){$0$};
\node at (5.4,.9){$\cong_{\partial}$};
\node at (3.6,1.3){$K$};
\node at (-0.1,0){$m_K$};
\node at (.35,.5){$0$};
\node at (9.1,1){$0$};
\node at (7.9,1.3){$K$};
\node at (6.1,.1){$m_K$};
 \end{tikzpicture}
    \caption{Our initial surgery diagram for $\partial W$ (left) is modified by a cancellation of the green curves to obtain a new diagram (right). Then perform a slam dunk on the $m_K$-framed curve to obtain the standard diagram for $S^3_{-1/m_K}(K)$.}
    \label{fig:mkneq0}
\end{figure}
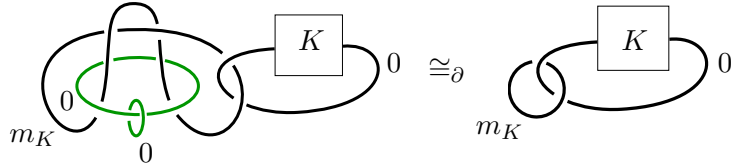

    \textbf{Case 2: $m_K= 0$.} First modify the surgery diagram for $Y(K,0)$ as shown in \cref{fig:mk0-first}, tracking the circle $\alpha$. 
    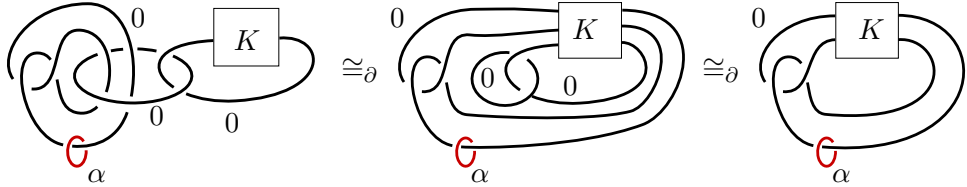
\begin{figure}[htb]
    \centering
      \begin{tikzpicture}
\node[anchor=south west,inner sep=0] at (0,0)
{\includegraphics[height=2.2cm]{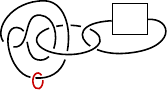}};
\node at (1.75,2){$0$};
\node at (2,0.7){$0$};
\node at (3.,0.6){$0$};
\node at (3.2,1.7){$K$};
\node at (4.7,1.3){$\cong_{\partial}$};
\node[anchor=south west,inner sep=0] at (5.2,0)
{\includegraphics[height=2.2cm]{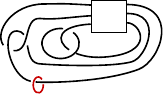}};
\node at (5.2,2){$0$};
\node at (6.4,1.2){$0$};
\node at (7.5,1.1){$0$};
\node at (7.7,1.85){$K$};
\node[anchor=south west,inner sep=0] at (10,0)
{\includegraphics[height=2.2cm]{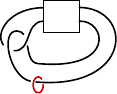}};
\node at (9.5,1.3){$\cong_\partial$};
\node at (10,2){$0$};
\node at (11.5,1.85){$K$};
\node at (6.3,-0.1){$\alpha$};
\node at (11.1,-0.1){$\alpha$};
\node at (1.2,-0.1){$\alpha$};
 \end{tikzpicture}
    \caption{Starting with \cref{fig:boundary} for $m_K=0$, we first do an isotopy to obtain the left diagram, then a handle slide to obtain the center diagram, then a slam dunk to obtain the right diagram, while tracking $\alpha$ throughout.}
    \label{fig:mk0-first}
\end{figure}
We obtain that $Y(K,0)\cong S^3_0(\Wh^+(K))$ via a diffeomorphism mapping~$\alpha$ to a meridian of $\Wh^+(K)$. Attach a $(-1)$-framed $2$-handle to $\partial X_D$ along $\alpha \subseteq Y(K,0)$. As before, the result is an integer homology ball, and  blowing down the  $(-1)$-framed curve shows that its boundary is $S^3_1(\Wh^+(K))$.
\end{proof}


\begin{proof}[Proof of \cref{thm:main}]

An integer homology sphere $Y$ that bounds an integer homology ball must have $d(Y)=0$, where $d$ is the Heegaard--Floer correction term of~\cite{OS03-dinv}.
Work of Ni--Wu~\cite{NiWu}*{Proposition 1.6} implies that $d(S^3_{1/n}(K))$ depends only on the sign of $n$, and so
 Proposition~\ref{prop:ZHB-bound} implies it will suffice to show that $d(S^3_{+1}(K)), d(S^3_{-1}(K))$, and $d(S^3_{+1}(\Wh^+(K))$ are all nonzero. 

Golla observed (see \cite{LidmanTweedy}*{Proof of Theorem 1.3}) that the cabling formulas from~\cite{Hom14-eps} and \cite{WuCabling} imply that $d(S^3_{+1}(K))$ and $ d(S^3_{-1}(K))$ are both nonzero. We now show that $d(S^3_{+1}(\Wh^+(K)))\neq 0$.
Hom and Wu defined a knot invariant called $\nu^+$, such that for any knot $J$ the following hold~\cite{HomWu}*{Proposition 2.3}:
\begin{enumerate}
    \item[(a)] $d(S^3_{+1}(J))=0$ if and only if $\nu^+(J)=0$; and 
    \item[(b)]$\nu^+(J) \geq \tau(J)$, where $\tau$ is the Heegaard--Floer knot invariant defined in~\cite{OS03-tauinv}. 
\end{enumerate} 
So we will be done if we can show that $\tau(\Wh^+(K))>0$, which by~\cite{Hedden07-Wh}*{Theorem 1.4} is true if and only if $\tau(K)>0$. 
We  now use \cite{Hom14-eps}*{Theorem 1} and the well-known computations of $\tau$ and $\varepsilon$ for torus knots to obtain
\begin{align*}
\tau(-K)=\tau(C_{2,4m-1}(T_{-2,6m+1}))
&= -2\frac{(2-1)(6m+1-1)}{2}+\frac{(2-1)(4m-1+1)}{2}=-4m.
\end{align*}
So $\tau(K)=4m>0$ as desired.
\end{proof}


One might hope for simpler examples of knots which are not smoothly RH-slice. In particular, it seems unclear whether the trefoil is smoothly RH-slice. We note that the statements of Proposition~\ref{prop:ZHB-bound} are certainly not the strongest conclusions one can draw from knowing that a knot is smoothly RH-slice, just ones that sufficed for our purpose of establishing an infinite family of non smoothly RH-slice knots. 
We therefore ask the following.
\begin{question}
    Is there a non-smoothly slice knot that is smoothly RH-slice?
\end{question}

Since the topological RH-slice problem is related to the topological $s$-cobordism theorem and topological surgery conjecture, it is also natural to ask if one could use the failure of the smooth RH-slice problem to construct new counterexamples in the smooth category.  For example, while the smooth $s$-cobordism conjecture is known to be false by Donaldson's work \cite{Donaldson87}, it would be interesting if the failure of smooth RH-slicing could lead to new constructions of exotica. 
\def\MR#1{}
\bibliography{bib}
\end{document}